\documentclass[sigplan,screen]{acmart}

\pdfoutput=1 

\setcopyright{acmlicensed}
\acmPrice{15.00}
\acmDOI{10.1145/3372885.3373826}
\acmYear{2020}
\copyrightyear{2020}
\acmISBN{978-1-4503-7097-4/20/01}
\acmConference[CPP '20]{Proceedings of the 9th ACM SIGPLAN International Conference on Certified Programs and Proofs}{January 20--21, 2020}{New Orleans, LA, USA}
\acmBooktitle{Proceedings of the 9th ACM SIGPLAN International Conference on Certified Programs and Proofs (CPP '20), January 20--21, 2020, New Orleans, LA, USA}

\usepackage{cleveref}
\usepackage{booktabs}   
\usepackage{subcaption} 

\usepackage{color}
\definecolor{keywordcolor}{rgb}{0.7, 0.1, 0.1}   
\definecolor{tacticcolor}{rgb}{0.1, 0.2, 0.6}    
\definecolor{commentcolor}{rgb}{0.4, 0.4, 0.4}   
\definecolor{symbolcolor}{rgb}{0.0, 0.1, 0.6}    
\definecolor{sortcolor}{rgb}{0.1, 0.5, 0.1}      
\definecolor{attributecolor}{rgb}{0.7, 0.1, 0.1} 

\makeatletter
\DeclareRobustCommand\widecheck[1]{{\mathpalette\@widecheck{#1}}}
\def\@widecheck#1#2{%
    \setbox\z@\hbox{\m@th$#1#2$}%
    \setbox\tw@\hbox{\m@th$#1%
       \widehat{%
          \vrule\@width\z@\@height\ht\z@
          \vrule\@height\z@\@width\wd\z@}$}%
    \dp\tw@-\ht\z@
    \@tempdima\ht\z@ \advance\@tempdima2\ht\tw@ \divide\@tempdima\thr@@
    \setbox\tw@\hbox{%
       \raise\@tempdima\hbox{\scalebox{1}[-1]{\lower\@tempdima\box
\tw@}}}%
    {\ooalign{\box\tw@ \cr \box\z@}}}
\makeatother

\usepackage{listings}

\usepackage{stmaryrd}
\newcommand{\B}{\mathbb{B}}
\newcommand{\lil}{\lstinline}
\newcommand{\N}{\mathbb{N}}
\newcommand{\ZFC}{\mathsf{ZFC}}
\newcommand{\CH}{\mathsf{CH}}

\usepackage{amsthm}
\newtheorem{thm}{Theorem}[section]
\theoremstyle{definition}
\newtheorem{defn}{Definition}[section]
\newtheorem{remark}{Remark}[section]
\DeclareMathOperator{\cf}{cf}
\DeclareMathOperator{\Ord}{Ord}

\usepackage{tikz-cd}

\begin{document}

\title{A Formal Proof of the Independence of the Continuum Hypothesis}


\author{Jesse Michael Han}
\affiliation{
  \department{Department of Mathematics}              
  \institution{University of Pittsburgh}            
  \streetaddress{4200 Fifth Ave}
  \city{Pittsburgh}
  \state{PA}
  \postcode{15260}
  \country{USA}                    
}
\email{jessemichaelhan@gmail.com}          

\author{Floris van Doorn}
\orcid{0000-0003-2899-8565}             
\affiliation{
  \department{Department of Mathematics}              
  \institution{University of Pittsburgh}            
  \streetaddress{4200 Fifth Ave}
  \city{Pittsburgh}
  \state{PA}
  \postcode{15260}
  \country{USA}                    
}
\email{fpvdoorn@gmail.com}         

\begin{abstract}
  We describe a formal proof of the independence of the continuum hypothesis (\(\mathsf{CH}\)) in the Lean theorem prover. We use Boolean-valued models to give forcing arguments for both directions, using Cohen forcing for the consistency of \(\neg \mathsf{CH}\) and a \(\sigma\)-closed forcing for the consistency of \(\mathsf{CH}\).
\end{abstract}




\maketitle

\section{Introduction}
\label{sect:intro}
The continuum hypothesis ($\mathsf{CH}$) states that there is no cardinality between $\omega$, the smallest infinite cardinal and $\mathfrak{c}$, the cardinality of the continuum.
It was posed by Cantor \cite{cantor1878beitrag} in 1878 and was the first problem on Hilbert's list of twenty-three unsolved problems in mathematics.
G\"odel \cite{godel1938consistency} proved in 1938 that $\mathsf{CH}$ was consistent with Zermelo-Fraenkel set theory with the axiom of choice ($\ZFC$). He conjectured that \(\mathsf{CH}\) was independent, i.e. neither provable nor disprovable, from \(\ZFC\). This remained an open problem until 1963, when Paul Cohen developed \emph{forcing} \cite{cohen-the-independence-of-the-continuum-hypothesis-1,cohen1964independence2} and used it to prove the consistency of $\neg \mathsf{CH}$ with \(\ZFC\), completing the independence proof. This work started modern set theory, and for his invention of forcing, Cohen was awarded a Fields medal.

The independence of \(\CH\) has also been an open formalization problem. Since 2005, Freek Wiedijk has maintained a list (\emph{Formalizing 100 theorems}~\cite{wiedijk100theorems}) of one hundred problems for formalized mathematics, with the independence of \(\CH\) as the 24th. As of 2019, it was one of the six remaining problems.

In this paper we describe the successful completion of the Flypitch project\footnote{\url{https://flypitch.github.io}} (\textbf{F}ormal\textbf{ly} \textbf{p}roving the \textbf{i}ndependence of \textbf{t}he \textbf{c}ontinuum \textbf{h}ypothesis).
We formalize forcing with Boolean-valued models. We use Cohen forcing to construct a Boolean-valued model of \(\ZFC\) where \(\CH\) is false, and a \(\sigma\)-closed forcing to construct a Boolean-valued model of \(\ZFC\) where \(\CH\) is true. We then combine this with a deep embedding of first-order logic, including a proof system and the axioms of \(\ZFC\), to verify that \(\CH\) is neither provable nor disprovable from \(\ZFC\).

Our formalization\footnote{\url{https://github.com/flypitch/flypitch}} uses the Lean 3 theorem prover, building on top of \textsf{mathlib}~\cite{mathlib}.
Lean is an interactive proof assistant under active development at Microsoft Research~\cite{de2015lean, ullrich2019counting}. It has a similar metatheory to Coq, adding definitional proof irrelevance, quotient types, and a noncomputable choice principle.
Our formalization makes as much use of the expressiveness of Lean's dependent type theory as possible, using constructions which are impossible or unwieldy to encode in HOL, let alone ZF.
The types of cardinals and ordinals in \textsf{mathlib}, which are defined as proper equivalence classes of (well-ordered) types, live one universe level higher than the types used to construct them, and our models of set theory require as input an entire universe of types. Our encoding of first-order logic also uses parameterized inductive types which ensure that type-correctness implies well-formedness, eliminating the need for separate well-formedness proofs.


The method of forcing with Boolean-valued models was developed by Solovay and Scott \cite{scott1967proof,scott-solovay} as a simplification of Cohen's method.
Some of these simplifications were incorporated by Shoenfield \cite{shoenfield1971unramified} into a general theory of forcing using partial orders, and it is in this form that forcing is usually practiced.
While both approaches have essentially the same mathematical content (see e.g.\ \cite{kunen2014set, jech2013set, moore2019method}), there are several reasons why we chose to use Boolean-valued models.
The main reason is the directness of forcing with Boolean-valued models, which bypasses the need for the L\"owenheim-Skolem theorems, Mostowski collapse, countable transitive models, or genericity considerations for filters.
The theory of forcing with Boolean-valued models also cleanly splits into several parts, allowing us to formalize different components in parallel (e.g. a general theory of Boolean-valued semantics, a library for calculations in complete Boolean algebras, a construction of Boolean-valued models of set-theory) and later recombine them. In particular, our library for Boolean-valued semantics for first-order logic is completely general and can be reused for other formalization projects. Finally, our Boolean-valued models of set theory are inductive types generalizing the Aczel encoding of set theory into dependent type theory; consequently, the automatically-generated induction principle \emph{is} \(\in\)-induction, leading to cleaner proofs.

\subsection{Proof Outline}
\label{subsect:intro:outline}

The usual method to show that a statement is unprovable is to construct a model where the statement is false, and apply the soundness theorem; our method is similar, except that we use Boolean-valued semantics and a Boolean-valued soundness theorem (see \Cref{sect:boolean-semantics}).
The difference between Boolean-valued models and ordinary models is that the truth values in a Boolean-valued model \lil{M} live in a complete Boolean algebra \lil{(𝔹, ⊓, ⊔, ⨅, ⨆,⊥,⊤)}.
If we can construct two Boolean-valued models of $\ZFC$, one where $\CH$ is true \(\top\), and one where $\CH$ is false \(\bot\), then by the Boolean-valued soundness theorem, $\CH$ is independent from $\ZFC$.

For any complete Boolean algebra $\mathbb{B}$ we implement the set-theoretic universe \(V^{\mathbb{B}}\) of \(\mathbb{B}\)-valued sets by generalizing the Aczel encoding of set theory (called \lil{pSet}, see \Cref{sect:bset}), obtaining a type \lil{bSet 𝔹} of $\mathbb{B}$-valued sets. The fundamental theorem of forcing for Boolean-valued models \cite{hamkins2012well}, translated to our situation, then states that \lil{bSet 𝔹} is a \lil{𝔹}-valued model is \(\ZFC\).

To show the independence of \(\CH\), it remains to construct two appropriate complete Boolean algebras 
The properties of \lil{bSet 𝔹} can vary wildly depending on the choice of the complete Boolean algebra \lil{𝔹}. There is always a map \lil{check : pSet → bSet 𝔹}, $x \mapsto \widecheck{x}$, but in general, $\widecheck{x}$ might have different properties than $x$. Making a good choice of \(\B\) and controlling the behavior of the check-names is precisely the task of forcing (\Cref{sect:forcing}).

Traditional presentations of forcing, even with Boolean-valued models (e.g. \cite{bell2011set}, \cite{jech2013set}), are careful to stay within the foundations of $\ZFC$, emphasizing that all arguments may be performed internal to a model of $\ZFC$, etc. In order to formalize these set-theoretic arguments in a type-theoretic metatheory, it is important to separate their mathematical content from their metamathematical content. It is not immediately clear what parts of these arguments use their set-theoretic foundation in an essential way and require modification in the passage to type theory. Our formalization clarifies some of these questions.

We use custom domain-specific tactics and various forms of automation throughout our formalization, notably a tactic library for simulating natural deduction proofs inside a complete Boolean algebra (\Cref{sect:metaprogramming}).
This reveals another advantage of working in a proof assistant: the bookkeeping of Boolean truth-values, sometimes regarded as a tedious aspect of the Boolean-valued approach to forcing, can be automated away.

\paragraph{Contributions}
An earlier paper \cite{DBLP:conf/itp/HanD19} describes a formalization of Cohen forcing and the unprovability of \(\CH\). In order to keep our presentation self-contained, 
we reproduce some of that material here, incorporating it into our discussions of our deep embedding of first-order logic/Boolean-valued semantics, usage of metaprogramming, and the Cohen forcing argument.
Our main novel contribution is a formalization of collapse forcing and the unprovability of \(\neg \CH\), thereby providing the first formalization of the independence of \(\CH\) in a single theorem prover. For reasons we will see in \Cref{sect:forcing}, the forcing argument for \(\CH\) requires far more set theory and is harder to formalize than the forcing argument for \(\neg \CH\). Moreover, we elaborate on parts of the formalization which were omitted from \cite{DBLP:conf/itp/HanD19}, including expanded discussions of our implementation of the \(\ZFC\) axioms and our formalization of the \(\Delta\)-system lemma.

\paragraph{Sources}
Our strategy for forcing \(\neg \CH\) is a synthesis of the proofs in the textbooks of Bell (\cite{bell2011set}, Chapter 2) and Manin (\cite{manin2009course}, Chapter 8).
For the $\Delta$-system lemma, which we use to verify that Cohen forcing is CCC, we follow Kunen (\cite{kunen2014set}, Chapters 1 and 5).

We were unable to find a reference for a purely Boolean-valued account of forcing \(\CH\). We loosely followed the conventional arguments given by Weaver (\cite{weaver2014forcing}, Chapter 12) and Moore (\cite{moore2019method}), and base our construction of \(\mathbb{B}_{\mathsf{collapse}}\) on the collapse algebras defined by Bell (\cite{bell2011set}, Exercise 2.18).

\paragraph{Related Work} Set theory and first-order logic are both common targets for formalization. Shankar \cite{shankar1997metamathematics} used a deep embedding of first-order logic for incompleteness theorems. Harrison gives a deeply-embedded implementation of first-order logic in HOL Light \cite{harrison1998formalizing} and a proof-search style account of the completeness theorem in \cite{harrison2009handbook}. Other formalizations of first-order logic can be found in Isabelle/HOL (\cite{Ridge2005AMV}, \cite{schlichtkrull2018formalization},\cite{FOL-Fitting-AFP}) and Coq (\cite{ilik2010constructive}, \cite{DBLP:conf/tphol/OConnor05}).

A large body of formalized set theory has been completed in Isabelle/ZF, led by Paulson and his collaborators \cite{paulson1996mechanizing, paulson1993set, paulson2002reflection}, including the relative consistency of \(\mathsf{AC}\) with $\mathsf{ZF}$ \cite{paulson2008relative}. Building on this, Gunther, Pagano, and Terraf have taken some first steps towards formalizing forcing \cite{gunther2018first, gunther2019mechanization}, by way of generic extensions of countable transitive models.



\section{First-Order Logic}
\label{sect:fol}

The starting point for first-order logic is a \emph{language} of relation and function symbols.
We represent a language as a pair of $\N$-indexed families of types, each of which is to be thought of as the collection of relation (resp. function) symbols stratified by arity:
\begin{lstlisting}
structure Language : Type (u+1) :=
(functions : ℕ → Type u)
(relations : ℕ → Type u)
\end{lstlisting}
\subsection{Terms, Formulas and Proofs}
\label{subsect:fol:terms}
The main novelty of our implementation of first-order logic is the use of \emph{partially applied} terms and formulas, encoded in a parameterized inductive type where the $\N$ parameter measures the difference between the arity and the number of applications.
The benefit of this is that it is impossible to produce an ill-formed term or formula, because type-correctness is equivalent to well-formedness.
This eliminates the need for separate well-formedness proofs.

Fix a language $L$. We define the type of \textbf{preterms} as follows:
\begin{lstlisting}
inductive preterm (L : Language.{u}) :
    ℕ → Type u
| var : ℕ → preterm 0 -- notation `&`
| func {l : ℕ} : L.functions l → preterm l
| app {l : ℕ} :
    preterm (l + 1) → preterm 0 → preterm l
\end{lstlisting}
A member of \lil{preterm n} is a partially applied term.
If applied to \lil{n} terms, it becomes a term.
We define the type of well-formed terms \lil{term L} to be \lil{preterm L 0}. 


The type of \textbf{preformulas} is defined similarly:
\begin{lstlisting}
inductive preformula (L : Language.{u}) :
    ℕ → Type u
| falsum : preformula 0 -- notation ⊥
| equal : term L → term L → preformula 0
    -- notation ≃
| rel {l : ℕ}, L.relations l → preformula l
| apprel {l : ℕ}, preformula (l + 1) →
    term L → preformula l
| imp : preformula 0 → preformula 0 →
    preformula 0 -- notation ⟹
| all : preformula 0 → preformula 0
    -- notation ∀'
\end{lstlisting}
We choose this definition of \lil{preformula} to mimic \lil{preterm}.
A member of \lil{preformula n} is a partially applied formula, and if applied to \lil{n} terms, it becomes a formula.
The type of well-formed formulas \lil{formula L} is defined to be \lil{preformula L 0}.
Implication is the only primitive binary connective and universal quantification is the only primitive quantifier. Since we use classical logic, we can define the other connectives and quantifiers from these.
Note that implication and the universal quantifier cannot be applied to preformulas that are not fully applied.

It is also possible to define well-typed terms and formulas using vectors of terms and nested inductive types. However, we avoided these kinds of definitions because Lean has limited support for nested inductive types. In the case of formulas, this would not even result in a nested inductive type, but we found it more convenient to adapt operations and proofs from \lil{preterm} to \lil{preformula} using our definition.

We use de Bruijn indices to avoid variable shadowing. This means that the variable \lil{&m} under \lil{k} is bound if $m<k$ and otherwise represents the $(m-k)$-th free variable.
We define the usual operations of lifting and substitution for terms and formulas, needed when using de Bruijn variables.
The notation \lil{t ↑' n # m} means the preterm of preformula \lil{t} where all variables which are at least \lil{m} are increased by \lil{n}.
The lift \lil{t ↑' n # 0} is abbreviated to \lil{t ↑ n}.
The substitution \lil{t[s // n]} is defined to be the term or formula \lil{t} where all variables that represent the \lil{n}-th free variable are replaced by \lil{s}.
More specifically, if an occurrence of a variable \lil{&(n+k)} is under \lil{k} quantifiers, then it is replaced by \lil{s ↑ (n+k)}.
Variables \lil{&m} for $m>n+k$ are replaced by \lil{&(m-1)}.

Our proof system is a natural deduction calculus, and all rules are motivated to work well with backwards-reasoning. The type of proof trees is given by the following inductive family of types:
\label{def:prf}
\begin{lstlisting}
inductive prf :
    set (formula L) → formula L → Type u
| axm Γ A : A ∈ Γ → prf Γ A
| impI Γ A B : prf (insert A Γ) B →
    prf Γ (A ⟹ B)
| impE Γ A B : prf Γ (A ⟹ B) → prf Γ A →
    prf Γ B
| falsumE Γ A : prf (insert ∼A Γ) ⊥ → prf Γ A
| allI Γ A : prf ((λ f, f ↑ 1) '' Γ) A →
    prf Γ (∀' A)
| allE₂ Γ A t : prf Γ (∀' A) →
    prf Γ (A[t // 0])
| ref Γ t : prf Γ (t ≃ t)
| subst₂ Γ s t f : prf Γ (s ≃ t) →
    prf Γ (f[s // 0]) → prf Γ (f[t // 0])
\end{lstlisting}
In \lil{allI} the notation \lil{(λ f, f ↑ 1) '' Γ} means lifting all free variables in \lil{Γ} by one.
A term of type \lil{prf Γ A}, denoted \lil{Γ ⊢ A}, is a proof tree encoding a derivation of $A$ from $\Gamma$.
We also define provability as the proposition stating that a proof tree exists.

\noindent \begin{minipage}{\columnwidth} 
\begin{lstlisting}
def provable (Γ : set (formula L))
  (f : formula L) : Prop := nonempty (prf Γ f)
\end{lstlisting}
\end{minipage}
Our current formalization does not use the data of proof trees in an essential way, but we defined them so that we can define manipulations on proof trees (like detour elimination) in future projects.
Besides Boolean-valued semantics (\Cref{sect:boolean-semantics}), we also formalize ordinary first-order semantics, and our work includes a formalization of the completeness (and compactness) theorems using Henkin term models.

\subsection{ZFC}
\label{subsect:fol:zfc}

Usually, the language of set theory has one binary relation symbol and no function symbols.
To make the language easier to work with, and to concisely formulate the continuum hypothesis, we conservatively extend \(\ZFC\) with the following function symbols: the empty set \(\varnothing\), ordered pairing \(({-},{-})\), the natural numbers \(\omega\), power set \(\mathcal{P}({-})\) and union \(\bigcup({-})\).
This gives a conservative extension of the regular theory of ZFC, because these function symbols are all definable.
\begin{figure*}
  \begin{center}
  \begin{minipage}{0.86\textwidth}
  \begin{lstlisting}
  axiom_of_emptyset := ∀ x, x ∉ ∅
  axiom_of_ordered_pairs := ∀ x y z w, (x, y) = (z, w) ↔ x = z ∧ y = w
  axiom_of_extensionality := ∀ x y, (∀ z, (z ∈ x ↔ z ∈ y)) → x = y
  axiom_of_union := ∀ u x, x ∈ ⋃ u ↔ ∃ y ∈ u, x ∈ y
  axiom_of_powerset := ∀ z y, y ∈ P(z) ↔ ∀ x ∈ y, x ∈ z
  axiom_of_infinity := ∅ ∈ ω ∧ (∀ x ∈ ω, ∃ y ∈ ω, x ∈ y) ∧ (∃ α, Ord(α) ∧ ω = α) ∧
    ∀ α, Ord(α) → (∅ ∈ α ∧ ∀ x ∈ α, ∃ y ∈ α, x ∈ y) → ω ⊆ α
  axiom_of_regularity := ∀ x, x ≠ ∅ → ∃ y ∈ x, ∀ z ∈ x, z ∉ y
  zorns_lemma := ∀ z, z ≠ ∅ → (∀ y, (y ⊆ z ∧ ∀ x₁ x₂ ∈ y, x₁ ⊆ x₂ ∨ x₂ ⊆ x₁) → (⋃y) ∈ z) →
    ∃ m ∈ x, ∀ x ∈ z, m ⊆ x → m = x
  axiom_of_collection(ϕ) := ∀ p ∀ A, (∀ x ∈ A, ∃ y, ϕ(x,y,p)) →
    (∃ B, (∀ x ∈ A, ∃ y ∈ B, ϕ(x,y,p)) ∧ ∀ y ∈ B, ∃ x ∈ A, ϕ(x,y,p))

  epsilon_transitive(z) := ∀ x, x ∈ z ⟹ x ⊆ z
  epsilon_trichotomy(z) := ∀ x y ∈ z, x = y ∨ x ∈ y ∨ y ∈ x
  epsilon_wellfounded(z) := ∀ x, x ⊆ z ⟹ x ≠ ∅ → ∃ y ∈ x, ∀ w ∈ x, w ∉ y
  Ord(z) := epsilon_trichotomy(z) ∧ epsilon_wellfounded(z) ∧ epsilon_transitive(z)
  \end{lstlisting}
  \end{minipage}
  \caption{Our formulation of \(\mathsf{ZFC}\).}
  \label{figure:ZFC}
  \end{center}
\end{figure*}

In Figure~\ref{figure:ZFC} we have listed all the axioms of \(\ZFC\) written using names variables (the formalization uses de Bruijn variables). We also include the definition of ordinal, which is used in the axiom of infinity. Note that \lstinline{epsilon_wellfounded} follows for every set from the axiom of regularity, but we add it for the sake of completeness. The only axiom scheme is \lstinline{axiom_of_collection} which ranges over all formulas \lstinline{ϕ(x,y,p)} with (at most) \lstinline{n+2} free variables, where \lstinline{p} is a vector of length \lstinline{n}.




Now \(\CH\) is defined to be the sentence
\[\CH:=\forall x, \Ord(x) \Rightarrow x \le \omega \vee \mathcal{P}(\omega) \le x,\]
where \(x \le y\) means that there is a surjection from a subset of \(y\) to \(x\). In code, we have:
\begin{lstlisting}
def CH_formula : formula L_ZFC :=
∀' (is_ordinal ⟹
  leq_f[omega_t//1] ⊔ leq_f[Powerset_t omega_t//0])
\end{lstlisting}
The substitutions ensure that the formulas are applied to the correct arguments, and \lil{⊔} is notation for disjunction.

\section{Boolean-Valued Semantics}
\label{sect:boolean-semantics}
A \textbf{complete Boolean algebra} is a Boolean algebra $\B$ with additional operations infimum $(\bigsqcap)$ and supremum $(\bigsqcup)$ of any subset of $\B$.
We use $\sqcap, \sqcup, \implies, \top$, and $\bot$ to denote meet, join, material implication, top, and bottom. For more details on complete Boolean algebras, we refer the reader to the textbook of Halmos-Givant~\cite{givant2008introduction}.

\begin{defn}\label{def:boolean-valued-structure}
  Fix a language $L$ and a complete Boolean algebra $\B$. A \textbf{$\B$-valued structure} (or \lil{bStructure L 𝔹}) is a type $M$ equipped with the following.
  \begin{itemize}
    \item for every $n$-ary function symbol in a map $M^n \to M$;
    \item for every $n$-ary relation symbol a map $M^n \to \B$;
    \item a function ${\approx}:M\to M\to\B$ that is a Boolean valued congruence relation. This means that e.g.\
    $x\approx y\sqcap y\approx z\le x\approx z$ and that \[\bigsqcap_i x_i\approx y_i \le f(\vec x)\approx f(\vec y).\]
    There are similar conditions for reflexivity, symmetry and congruence for relation symbols.
  \end{itemize}
\end{defn}

Given a preterm \lil{t} in the language, we can realize it in any $\B$-valued structure $M$.
For this, we need to know the free variables in \lil{t}.
To do this conveniently with de Bruijn variables, we say that a (pre)term \lil{t} is \emph{bounded by \lil{l}} if all free variables are less than \lil{l} (i.e. all variables under \lil{k} quantifiers are less than \lil{k+l}).
Given \lil{t : preterm n} which is bounded by \lil{l}, and a realization \lil{v : vector M l} of the free variables, we define the realization $\llbracket t \rrbracket_M^v : M^n \to M$ by structural recursion on $t$.

For a formula $\varphi$ we do the same: we define bounded (pre)formulas, and define an realization $\llbracket \varphi \rrbracket_M^v : M^n \to \B$ by structural recursion.
If $\varphi$ is a sentence, the realization in a structure is just an element of the Boolean algebra: $\llbracket \varphi \rrbracket_M : \B$.

Since the truth values in a Boolean-valued model live inside the Boolean algebra $\B$ instead of just being true or false, we have to take a little care when stating the soundness theorem for Boolean-valued models.
Usually, a soundness theorem states something like ``if $\varphi$ is provable from hypotheses in $C$ then in every model where $C$ holds, $\varphi$ also holds.''
With Boolean truth-values, this is instead stated as an inequality of truth values. 
\begin{defn}
For $\Gamma : \B$ and a $\B$-valued structure $M$ we say that
\textbf{$\Gamma$ forces a sentence $\varphi$ in $M$}, written $\Gamma \Vdash_M \varphi$, if
$\Gamma \le \llbracket \varphi \rrbracket_M$.
We say that a set of sentences $C$ \textbf{models} $\varphi$, written $C \models_{\B}\varphi$, if for all non-empty $\B$-valued structures $M$ we have $\big(\bigsqcap_{\psi\in C}\big \llbracket \psi \rrbracket_M)\Vdash_M \varphi$.
\end{defn}
Using this definition, we can now state the Boolean-valued soundness theorem: \label{boolean-soundness}
\begin{lstlisting}
theorem boolean_soundness {Γ : set (sentence L)}
    {ϕ : sentence L} : Γ ⊢ ϕ → Γ ⊨[𝔹] ϕ
\end{lstlisting}
The proof is a straightforward structural induction.

\section{Boolean-Valued Models of Set Theory}
\label{sect:bset}
\subsection{The Aczel Encoding}
\label{subsect:bset:aczel}
Our starting point is the Aczel encoding of \(\ZFC\) (\cite{aczel1978type, aczel1986type, aczel1982type}) into dependent type theory.
This was implemented in Coq by Werner \cite{werner1997sets}, and in Lean's \textsf{mathlib} by Carneiro \cite{mario1}.
The idea is to take a type universe \lstinline{Type u} and imitate the cumulative hierarchy construction with an inductive type:
\begin{lstlisting}
inductive pSet : Type (u+1)
| mk (α : Type u) (A : α → pSet) : pSet
\end{lstlisting}
For an element \lil{x = ⟨α, A⟩ : pSet}, the function \lil{A} points to the elements of \lil{x}.
We can define the empty set as \lstinline{∅ := ⟨empty, empty.elim⟩ : pSet}.
Note that \lil{pSet} does not satisfy the axiom of extensionality.
In order to obtain a model where the axiom of extensionality holds, we must quotient \lstinline{pSet} by \emph{extensional equivalence}:
\begin{lstlisting}
def equiv : pSet → pSet → Prop
| ⟨α,A⟩ ⟨β,B⟩ := (∀a, ∃b, equiv (A a) (B b)) ∧
  (∀b, ∃a, equiv (A a) (B b))
\end{lstlisting}
One can then define membership from equivalence and check that modulo extensional equivalence, \lstinline{pSet} is a model of \(\ZFC\).

\subsection{Boolean-Valued Sets}
\label{subsect:bset:bset}

We now want to generalize \lil{pSet} to a Boolean-valued model of \(\ZFC\). We must give a
\(\B\)-valued predicate interpreting the membership symbol \lil{∈}. We will encode this information by extending each \lil{⟨α,A⟩ : pSet} with an additional function \lil{B : α → 𝔹}, which has the effect of attaching a \emph{Boolean truth-value} to every element of \lil{⟨α,A⟩}:
\begin{lstlisting}
inductive bSet (𝔹 : Type u)
    [complete_boolean_algebra 𝔹] : Type (u+1)
| mk (α : Type u) (A : α → bSet)
    (B : α → 𝔹) : bSet
\end{lstlisting}
The \lil{𝔹}-valued predicate \lil{B} expresses that \lil{A a ∈ ⟨α, A, B⟩} has truth value (at least) \lil{B i}. For convenience, if \lstinline{x : bSet 𝔹} and \lstinline{x := ⟨α, A, B⟩}, we put \lstinline{x.type := α, x.func := A, x.bval := B}.

One can also be led to this construction by considering the recursive \emph{name}-construction from forcing, a key ingredient to building forcing extensions. Let \(\mathbb{P}\) be a poset. From e.g. (Kunen \cite{kunen2014set}, Definition IV.2.5):
\begin{defn}
  A set \(\tau\) is a \(\mathbb{P}\)-name iff \(\tau\) is a relation and for all \(\langle  \sigma, p\rangle \in \tau\) we have that \(\sigma\) is a \(\mathbb{P}\)-name and \(p \in \mathbb{P}\).
\end{defn}

In particular, if \(\mathbb{P}\) is the singleton poset, then a \(\mathbb{P}\)-name is merely a set of \(\mathbb{P}\)-names, in the same way that a term of type \lstinline{pSet} is a type-indexed collection of terms of type \lstinline{pSet}.
Reversing this observation, we can replace \(\mathbb{P}\) with a complete Boolean algebra \(\mathbb{B}\) and generalize the definition of \lstinline{pSet.mk} with a third field, so that as in the case of \(\mathbb{P}\)-names, every element of a set is assigned an element (a ``Boolean truth-value'') of \(\mathbb{B}\), again giving us \lil{bSet 𝔹}. Thus, \lil{bSet 𝔹} should be thought of as the type of \lil{𝔹}-names.


\paragraph{Boolean-Valued Equality and Membership}
We can define Boolean-valued equality and membership analogously to the definitions in \lil{pSet}.
To do this, we translate quantifiers and connectives into operations on $\B$:
\begin{lstlisting}
def bv_eq : bSet 𝔹 → bSet 𝔹 → 𝔹
| ⟨α, A, B⟩ ⟨α', A', B'⟩ :=
  (⨅a, B a ⟹ ⨆a', B' a' ⊓ bv_eq (A a) (A' a')) ⊓
  (⨅a', B' a' ⟹ ⨆a, B a ⊓ bv_eq (A a) (A' a'))
\end{lstlisting}
We abbreviate \lil{bv_eq} with the infix operator \lil{=ᴮ}.
It is now easy to define $\B$-valued membership, which we denote by \lil{∈ᴮ}.
\begin{lstlisting}
def mem : bSet 𝔹 → bSet 𝔹 → 𝔹
| x ⟨α, A, B⟩ := ⨆a, B a ⊓ x =ᴮ A a
\end{lstlisting}
While standard treatments of Boolean-valued models of \(\ZFC\) mutually define equivalence and membership so that the axiom of extensionality follows definitionally (\cite{bell2011set}, \cite{hamkins2012well}), the induction principle given by the non-mutual definition is easier to work with in our formalization.

\subsection{The Fundamental Theorem of Forcing}
\label{subsect:bset:fundamental-thm}
The fundamental theorem of forcing for Boolean-valued models~\cite{hamkins2012well} states that for any complete Boolean algebra \lil{𝔹}, the type \lil{bSet 𝔹} forms a Boolean-valued model of $\ZFC$.


We mostly follow Bell~\cite{bell2011set} for the verification of the \(\ZFC\) axioms in \lil{bSet 𝔹}.
Although most of the argument is routine, we describe some aspects of \lil{bSet 𝔹} which are revealed by this verification.

Notably, we can define subsets of a set \lil{x : bSet 𝔹} by just modifying \lil{x.bval}.
This gives a nice definition of powerset:
\begin{defn} \label{def:powerset}
  Fix a $\B$-valued set \lil{x = ⟨α, A, b⟩} and \lil{χ : α → 𝔹} be a function.
  We define the $\B$-valued set $\widetilde{\chi}$ as \lil{⟨α, A, χ⟩}.
  The \textbf{powerset} $\mathcal{P}(x)$ of $x$ is defined to be the \lil{𝔹}-valued set
  \[\text{ \lstinline{set_of_indicator χ :=} }\langle \alpha \to \B, (\lambda\; \chi, \widetilde{\chi}),\ (\lambda\;\chi, \widetilde{\chi} \subseteq^B x)\rangle.\]
\end{defn}

In particular, this gives an easy implementation of the axiom of comprehension (not just for interpretations of formulas, but for any \(\mathbb{B}\)-valued predicate on \lil{bSet 𝔹} satisfying an appropriate \(\mathbb{B}\)-valued congruence lemma): \label{def:comprehension} 
\begin{lstlisting}
lemma bSet_axiom_of_comprehension (ϕ : bSet 𝔹 → 𝔹) (x : bSet 𝔹)
  (H_congr : B_ext ϕ) {Γ : 𝔹} :
  Γ ≤ ⨆ y, y ⊆ᴮ x ⊓ ⨅ z, z ∈ᴮ y ⇔ (z ∈ᴮ x ⊓ ϕ z)
\end{lstlisting}




Following Bell, we verify Zorn's lemma in \lil{bSet 𝔹}.
As is the case with \lil{pSet}, establishing Zorn's lemma requires the use of a choice principle from the metatheory.
This was the hardest part of our verification of the fundamental theorem of forcing, and relies on the technical tool of \emph{mixtures}, which allow sequences of $\B$-valued sets to be ``averaged'' into new ones. Using mixtures, one derives the \emph{maximum principle}, which allows existentially quantified statements to be instantiated without changing their truth-value (so is essentially the axiom of choice):
\begin{lstlisting}
lemma maximum_principle (ϕ : bSet 𝔹 → 𝔹)
(h_congr : B_ext ϕ) : ∃ u, (⨆(x:bSet 𝔹), ϕ x) = ϕ u
\end{lstlisting}
For example, if \lil{x : bSet 𝔹} and \lil{ϕ} is a \lil{𝔹}-valued predicate, if we have that \lil{⊤ ≤ ⨆ j : x.type, ϕ x}, there may not actually be some \lil{j : x.type} which attains that supremum. However, the maximum principle ensures that a witness can be constructed via mixtures.


After we verify the (shallow) statements of all the axioms in \lil{bSet 𝔹}, the last step is to construct a \(\mathbb{B}\)-valued \lil{L_ZFC}-structure, called \lil{V 𝔹}, on \lil{bSet 𝔹}, and check that the interpretations of the axioms are \(\top\). This amounts to proving that the deeply embedded statements correspond to the shallowly embedded statements. This is trivial for the axioms, since it is true by reflexivity, but takes more work for the axiom scheme of collection. This proves the following theorem.
\begin{lstlisting}
theorem fundamental_theorem_of_forcing :
  ⊤ ⊩[V 𝔹] ZFC
\end{lstlisting}
\subsection{Ordinals}
\begin{defn}\label{def:check}
  We define the canonical map \lil{check : pSet → bSet 𝔹} by
  \begin{lstlisting}
def check : pSet → bSet 𝔹
| ⟨α,A⟩ := ⟨α, check ∘ A, (λ a, ⊤)⟩
  \end{lstlisting}
  We write $\widecheck{x}$ for \lil{check x}, and call it a \emph{check-name}.
  These are also known as \emph{canonical names}, as they are the canonical representation of standard two-valued sets inside a Boolean-valued model of set theory.\footnote{We were pleased to discover Lean's support for custom notation allowed us to declare the Unicode modifier character \texttt{U+030C} ($\widecheck{\hspace{1mm}}$) as a postfix operator for \texttt{check}.}
\end{defn}

In general, $\widecheck{x}$ might have different properties than $x$, but \(\Delta_0\) properties (i.e. those definable with only bounded quantification) are always preserved. Importantly, \lil{bSet 𝔹} thinks $\widecheck{\omega}$ is $\omega$. Notably, \lil{ω : pSet} is defined separately from \lil{ordinal.mk omega} (see below) as the finite von Neumann ordinals indexed by $\mathbb{N}$, so the underlying types of \(\omega\) and \(\widecheck{\omega}\) are exactly \(\N\). 

The treatment of ordinals in \lil{mathlib} associates a class of ordinals to every type universe, defined as isomorphism classes of well-ordered types. Lean's ordinals may be represented inside \lil{pSet} by defining a map \lil{ordinal.mk : ordinal → pSet} via transfinite recursion (indexing the von Neumann construction of ordinals). In pseudocode,
\begin{lstlisting}
def ordinal.mk : ordinal → pSet
| 0 := ∅
| succ ξ := pSet.succ (ordinal.mk ξ)
            -- i.e. (mk ξ ∪ {mk ξ})
| is_limit ξ := ⋃ η < ξ, (ordinal.mk η)
\end{lstlisting}
Working internally to any model \(M\) of \(\ZFC\), we can define the class \(\operatorname{Ord}(M)\) as the collection of transitive sets which are well-ordered by their membership relation. While \lil{ordinal.mk} actually induces an order-isomorphism of \lil{pSet}'s ordinals with Lean's ordinals, the map \[\text{\lil{check ∘ ordinal.mk : ordinal → bSet 𝔹}}\] generally fails to surject onto \lil{bSet 𝔹}'s ordinals (in general, these are mixtures of checked ordinals).

We summarize the relationship between the three ``large'' types currently in play:
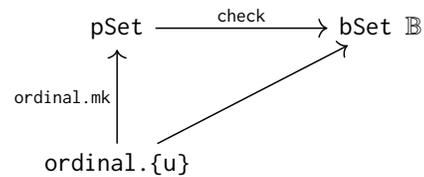
\[
  \begin{tikzcd}
    \texttt{pSet} \arrow{rr}{\texttt{check}} & & \texttt{bSet } \mathbb{B} & \\
    & & & \\
    \texttt{ordinal.\{u\}} \arrow{uu}{\texttt{ordinal.mk}} \arrow{uurr} & & & 
  \end{tikzcd}
\]

We adopt the convention to spell out the name of Lean ordinals and cardinals, and use (checked) Hebrew letters for their (Boolean-valued) set-theoretic counterparts, e.g.
\begin{lstlisting}
check (ordinal.mk (aleph 1)) = check (ℵ₁) = ℵ₁̌
\end{lstlisting}
We will freely conflate \lil{pSet} ordinals with their underlying types, so e.g. \(\nu : \aleph_2\) means \lil{ν : ℵ₂.type}. (It is always true that the cardinality of \lil{(ordinal.mk κ).type} is \lil{κ}.) Since in general, \(\widecheck{\aleph_1}\) is \emph{not} what \lil{bSet 𝔹} thinks is \(\aleph_1\), we will use a superscript, e.g. \(\aleph_k^{\mathbb{B}}\), to denote the internal alephs of \lil{bSet 𝔹}.


\section{Forcing} \label{sect:forcing}
Our point of departure from conventional accounts of forcing with a poset \(\mathbb{P}\) over a countable transitive model \cite{kunen2014set, jech2013set}, which use a generic filter to ``evaluate'' the \lstinline{ℙ}-names to produce an ordinary model of \(\ZFC\), is to force with \emph{Boolean-valued models} of \(\ZFC\) instead.
As first observed by Scott and Solovay \cite{scott-solovay}, this obviates the need for countable transitive models, generic filters, or the truth and definability lemmas, and allows us to work only with the \lstinline{𝔹}-names.

The cost of taking the \lstinline{𝔹}-names at face value is that the calculus of the forcing relation \cite{shoenfield1971unramified}, a key technical tool in usual forcing arguments, is replaced by the calculation of Boolean truth-values in \lstinline{𝔹}.
From the Boolean-valued perspective, forcing a sentence \(\Phi\) in the language of \(\ZFC\) means constructing some Boolean algebra \lil{𝔹} and a \lil{𝔹}-valued model \(M\) of \(\ZFC\) such that the truth value \(\Phi^{M}\) of \(\Phi\) is \(\top\).
We will always force over a type universe \lil{Type u}, and our Boolean-valued models of \(\ZFC\) are always of the form \lil{bSet 𝔹} for some \lil{𝔹 : Type u}.
That \lil{𝔹} belongs to the ``ground model'' \lil{Type u} is crucial for forcing, as specific choices of \lil{𝔹} will affect the structure of \lil{bSet 𝔹} (and hence the truth-value of \(\Phi\)). 

In this section, we describe two forcing arguments, one for \(\neg \CH\) and another for \(\CH\).
Both follow roughly the same pattern. In both cases, we require the existence of a function; for \(\neg \CH\), an injection \(\aleph_2 \hookrightarrow \mathcal{P}(\omega)\), and for \(\CH\), a surjection \(\aleph_1 \twoheadrightarrow \mathcal{P}(\omega)\).
We will construct a Boolean algebra \lil{𝔹} which encodes the construction (in \lil{Type u}) of such a function \(F\).
Then \lil{𝔹} induces in \lil{bSet 𝔹} an approximation \(\widetilde{F}\) to such a function, which \emph{a priori} is only between check-names.
To finish the forcing argument, we must show that it suffices to work with \(\widetilde{F}\).
This requires a careful study of how truth-values are calculated in \lil{bSet 𝔹},
and ultimately reduces to an analysis of how truth-values of \(\forall\)-\(\exists\) statements in \lil{bSet 𝔹} can be \emph{reflected back} to \lil{Type u}, and a verification of a combinatorial condition on \lil{𝔹}.

\subsection{Regular Open Algebras}
\begin{defn}\label{def:regular-open-algebra}
  Let $X$ be a topological space, and for any open set $U$, let $U^{\perp}$ denote the complement of the closure of $U$.
  The \textbf{regular open algebra} of a topological space $X$, written $\operatorname{RO}(X)$, is the collection of all open sets $U$ such that $U = (U^\perp)^\perp$, or equivalently such that
  $U$ is equal to the interior of the closure of $U$.
  $\operatorname{RO}(X)$ is equipped with the structure of a complete Boolean algebra, with $x \sqcap y := x \cap y$ and $x \sqcup y := ((x \cup y)^\perp)^\perp$ and $\neg x := x^\perp$ and $\bigsqcup x_i := ((\bigcup x_i)^\perp)^\perp$.
\end{defn}

While forcing conditions usually present themselves as a poset instead of a complete Boolean algebra, any forcing poset can be represented as the dense suborder of a regular open algebra \cite{moore2019method}.

\begin{defn}\label{def:dense-suborder}
  A \textbf{dense suborder} of \(\B\) is a subset \(\mathbb{P} \subseteq \B\) satisfying the following conditions: (1) for all \(p \in \mathbb{P}\), \(\bot < p\); (2) for all \(\bot < b \in \B\), there exists a \(p \in \mathbb{P}\) such that \(p \leq b\).
\end{defn}

We will use the following combinatorial conditions on \(\mathbb{B}\) in our forcing arguments:

\begin{defn}\label{def:ccc}
We say that $\B$ has the \textbf{countable chain condition} (CCC) if every antichain $\mathcal{A} : I \to \B$ (i.e. an indexed collection of elements $\mathcal{A} = \{a_i\}_i$ such that whenever $i \neq j, a_i \sqcap a_j = \bot$) has a countable image.
\end{defn}

\begin{defn}\label{def:sigma-closed}
We say that \(\B\) is \textbf{\(\sigma\)-closed} if there exists a dense suborder \(\mathbb{P}\) of \(\B\) such that every \(\omega\)-indexed downwards chain \(p_0 \geq \cdots \geq p_n \cdots\) in \(\mathbb{P}\) has a lower bound \(p_{\omega}\) in \(\mathbb{P}\).
\end{defn}

\subsection{Cohen Forcing}

As we have already seen in \Cref{def:powerset}, we construct the powerset of a \lil{𝔹}-valued set \lstinline{u : bSet 𝔹} using \lil{𝔹}-valued indicator functions \lil{χ : u.type → 𝔹}.
The basic strategy of Cohen forcing is to choose \lil{𝔹} such that for every \lil{ν : ℵ₂}, there is a canonical indicator function (a ``Cohen real'') \(\chi_{\nu} : \N \to \mathbb{B}\).
This is an external function (a member of a function type of \lstinline{Type u}) which descends to an injective function \(\widecheck{\aleph_2} \hookrightarrow \mathcal{P}(\omega)\) in \lil{bSet 𝔹}.

To show that the injection \(\widecheck{\aleph_2} \hookrightarrow \mathcal{P}(\omega)\)  suffices to negate \(\CH\), we will show that if \lil{𝔹} has the CCC, then \(\omega \prec \widecheck{\aleph_1} \prec \widecheck{\aleph_2}\), where $x\prec y$ means that there is no surjection from a subset of $x$ to $y$. We then ensure that \lil{𝔹} has this property by applying a powerful combinatorial argument called the \emph{\(\Delta\)-system lemma}.

\begin{defn}
  The \textbf{Cohen poset} for adding $\aleph_2$-many Cohen reals is the collection of all finite partial functions $\aleph_2 \times \mathbb{N} \to \mathbf{2}$, ordered by reverse inclusion.
\end{defn}

In the formalization, the Cohen poset is represented as a structure with three fields:
\begin{lstlisting}
structure ℙ_cohen : Type :=
  (ins : finset (ℵ₂.type × ℕ))
  (out : finset (ℵ₂.type × ℕ))
  (H : ins ∩ out = ∅)
\end{lstlisting}
That is, we identify a finite partial function \lil{f} with the triple \lil{⟨f.ins, f.out, f.H⟩}, where \lil{f.ins} is the preimage of $\{1\}$, \lil{f.out} is the preimage of $\{0\}$, and \lil{f.H} ensures that \lil{f} is well-defined.
While the members of the Cohen poset are usually defined as finite partial functions, we found that in practice \lil{f} is only needed to give a finite partial specification of a subset of $\aleph_2 \times \mathbb{N}$ (i.e. a finite set \lil{f.ins} which \emph{must} be in the subset, and a finite set \lil{f.out} which \emph{must not} be in the subset).
We chose this representation to make that information immediately accessible.

The Boolean algebra which we use for forcing $\neg\mathsf{CH}$ is
\[\mathbb{B}_{\mathsf{cohen}} := \operatorname{RO}(2^{\aleph_2 \times \mathbb{N}})\]
where we equip $2^{\aleph_2 \times \mathbb{N}}$ with the usual product space topology.

\begin{defn}
  We define the \textbf{canonical embedding} of the Cohen poset into \(\B_{\mathsf{cohen}}\) as follows:
\begin{lstlisting}
def ι : ℙ_cohen → 𝔹_cohen :=
λ p, {S | p.ins ⊆ S ∧ p.out ⊆ - S}
\end{lstlisting}
That is, we send each \(c : \mathbb{P}_{\mathsf{cohen}}\) to all subsets satisfying the specification given by \lil{c}. This is clopen, hence regular.
\end{defn}

Crucially, the image of this embedding is a dense suborder of \(\mathbb{B}_{\mathsf{cohen}}\).
This is essentially because the image of $\iota : \mathbb{P}_{\mathsf{cohen}} \to \B_{\mathsf{cohen}}$ \emph{is} the standard basis for the product topology.
Our chosen encoding of the Cohen poset also made it easier to perform this identification.

\begin{defn}\label{def:cohen-real}
  Let \(\nu : \aleph_2\). For any $n : \N$, the collection of all subsets of $\aleph_2 \times \N$ which contain $(\nu, n)$ is a regular open of $2^{\aleph_2 \times \N}$, denoted $\mathbf{P}_{(\nu, n)}$.
Thus, we associate to $\nu$ the $\B$-valued indicator function $\chi_{\nu} : \N \to \B$ defined by $\chi_{\nu}(n) := \mathbf{P}_{(\nu, n)}$.
  By \Cref{def:powerset}, each \(\chi_{\nu}\) induces a new $\B$-valued subset $\widetilde{\chi_{\nu}} \subseteq \widecheck{\N}$. We call $\widetilde{\chi_{\nu}}$ a \textbf{Cohen real}.
\end{defn}
\Cref{def:cohen-real} gives us an $\aleph_2$-indexed family of Cohen reals.
Converting this data into an injective function from \(\widecheck{\aleph_2}\) to $\mathcal{P}(\mathbb{N})$ inside \lil{bSet 𝔹} requires some care.
One must check that $\nu \mapsto \widetilde{\chi_{\nu}}$ is externally injective, and this is where the characterization of the Cohen poset as a dense subset of $\B$ (and moving back and forth between this representation and the definition as finite partial functions) comes in.

To finish negating \(\mathsf{CH}\), it suffices to show that \(\omega \prec \widecheck{\aleph_1} \prec \widecheck{\aleph_2}\),
i.e. that there is no surjection \(\widecheck{\omega} \twoheadrightarrow \widecheck{\aleph_1}\) and no surjection \(\widecheck{\aleph_1} \twoheadrightarrow \widecheck{\aleph_2}\).
We describe how we proved the latter claim; an identical argument can be used to show the former.

The strategy of the proof is to assume that there is a surjection \(\widecheck{\aleph_1} \twoheadrightarrow \widecheck{\aleph_2}\).
This surjectivity assumption is a Boolean-valued \(\forall\)-\(\exists\) statement about check-names, and we will \emph{reflect} it into the metatheory, producing a \(\forall\)-\(\exists\) statement about the non-checked counterparts in \lil{pSet}.
We will then use the CCC, a combinatorial condition on \(\B_{\mathsf{cohen}}\), to show that the reflected \(\forall\)-\(\exists\) statement implies a contradiction.

Specifically, we use the following lemma, which is true for general \(\mathbb{B}\):
\begin{lstlisting}
lemma AE_of_check_larger_than_check {x y : pSet}
  (f : bSet 𝔹) {Γ : 𝔹} (H_nonzero : ⊥ < Γ)
  (H : Γ ≤ is_surj_onto x̌ y̌ f) (Hy : ∃ z, z ∈ y) :
  ∀ i : y.type, ∃ j : x.type,
  ⊥ < is_func f ⊓ pair (x.func j)̌  (y.func i)̌  ∈ᴮ f
\end{lstlisting}
Suppose that there is a surjection \(\widecheck{\aleph_1} \twoheadrightarrow \widecheck{\aleph_2}\).
Applying this lemma to \(x := \widecheck{\aleph_1}\), \(y := \widecheck{\aleph_2}\), we obtain a \(\forall\)-\(\exists\) statement in the metatheory to which we can apply Lean's axiom of choice to produce a function \(g : \aleph_2 \to \aleph_1\).
Since externally, we know that \(\aleph_1 \prec \aleph_2\), it follows from the infinite pigeonhole principle that \(g\) must have an uncountable fiber over some \(\nu < \aleph_1\).
For every \(\eta \in g^{-1}(\{\nu\})\), let \(A_{\eta}\) be the element of \(\B_{\mathsf{cohen}}\) given by the lemma, i.e.
\[\text{\lil{(is_func f) ⊓ (pair (ℵ₁.func ν)̌  (ℵ₂.func η)̌  ∈ᴮ f)}}.\]
Because each \(A_{\eta}\) has as a conjunct the knowledge that \(f\) is a function, for \(\eta_1 \neq \eta_2\), \(A_{\eta_1}\) and \(A_{\eta_2}\) are incompatible, i.e. \(A_{\eta_1} \sqcap A_{\eta_2} = \bot\).
Since the lemma guarantees that each \(A_{\eta}\) is nonzero, the \(A_{\eta}\) form an uncountable antichain.
Therefore, if \(\mathbb{B}\) has the CCC, there is a contradiction. By \Cref{lemma:cohen-algebra-CCC}, \(\neg\mathsf{CH}\) is forced true in \(\text{\lil{bSet}}\ \mathbb{B}_{\mathsf{cohen}}\).

In our formalization, we actually prove a more general version of this argument, replacing \(\aleph_1\) and \(\aleph_2\) with any two infinite regular cardinals \(\kappa_1 < \kappa_2\).

\paragraph{CCC and the \(\Delta\)-system lemma}
To show that \(\B_{\mathsf{cohen}}\) has the CCC, we formalize and then apply a general result in transfinite combinatorics called the \emph{$\Delta$-system lemma}.
Though only briefly mentioned in \cite{DBLP:conf/itp/HanD19}, this was one of the most involved parts of our formalization of Cohen forcing, as it was a technical result in infinitary combinatorics. 
The details of the full argument are too technical to give here, so we omit the proofs in this section.


A family $(A_i)_i$ of sets is called a \textbf{$\Delta$-system} if there is a set $r$, called the \textbf{root} such that whenever $i \ne j$ we have $A_i \cap A_j = r$.
We write $c^{<\kappa}$ for the supremum of $c^\rho$ for $\rho<\kappa$.

\begin{lemma}[\(\Delta\)-system lemma (Theorem 1.6, \cite{kunen2014set})]\label{lemma:delta-system-lemma:general}
  Let \(\kappa\) be an infinite cardinal and let \(\theta > \kappa\) be regular,
  such that for all \(\alpha < \theta$ we have $\alpha^{<\kappa} < \theta\).
  For any family \(\{A_i\}_{i\in I}\) such that \(|I| \geq \theta\) and for all \(i\), \(|A_i| < \kappa\),
  there is a subfamily of size \(\theta\) which forms a \(\Delta\)-system.
\end{lemma}
The formalization closely follows the proof given in Kunen \cite[Chapter 2, Theorem 1.6]{kunen2014set}.
The proof involves tricky reasoning steps involving ordinals, which are common in infinitary combinatorics.
It starts by assuming that without loss of generality $\bigcup_i A_i\subseteq\theta$, so that all the $A_i$ are well-ordered, and by assuming that all $A_i$ have the same order-type.
These simplifying assumptions are harder to formalize, because that involves actually proving the general case from the special case.
It also involves defining a sequence by transfinite recursion, while simultaneously proving that the sequence has certain properties (lies below $\theta$).

In the formalization, the fact that the type of ordinals is a large type, i.e.\ lives one universe level higher than the types it is built from, causes difficulties. (These difficulties were also present earlier, because whenever we use e.g. ``\lil{ℵ₂.type}'', we are actually referring to a nonconstructively chosen witness for the order type of all the ordinals less than \lil{aleph 2}.)
The reason is that the original proof heavily uses sets of ordinals, and taking their order types, but in Lean this would involve calculating in both \lstinline"ordinal.{u}" and \lstinline"ordinal.{u+1}".
Instead, we frequently work with well-orders of a given order type, instead sets of ordinals, to do all computations in \lstinline"ordinal.{u}".

Lastly, one must take care to formulate the $\Delta$-system so that $\{A_i\}_i$ is an indexed family, instead of a collections of sets.
\Cref{thm:product-ccc} below does not follow conveniently from the $\Delta$-system lemma if it is formulated with a collection of sets; \cite{kunen2014set} is somewhat ambiguous about which version is used.

Setting \(\kappa = \omega\) and \(\theta = \aleph_1\) in \Cref{lemma:delta-system-lemma:general} yields:
\begin{lemma}\label{lemma:delta-system-lemma:simple}
  Any uncountable family of finite sets has an uncountable subfamily forming a $\Delta$-system.
\end{lemma}
We say that a topological space has the CCC if every family of pairwise disjoint open sets is countable. The proof of the following can be found in \cite{DBLP:conf/itp/HanD19}.
\begin{thm}\label{thm:product-ccc}
  For any family $(X_i)_{i\in I}$ of topological spaces, $\prod_{i\in I} X_i$ has the CCC if for every finite $J\subseteq I$, the product $\prod_{i\in J} X_i$ has the CCC.
\end{thm}
From \Cref{thm:product-ccc} and the observation that $2^J$ has the CCC if $J$ is finite, the result follows.
\begin{lemma}\label{lemma:cohen-algebra-CCC}
  \(\B_{\mathsf{cohen}}\) has the CCC.
\end{lemma}

\subsection{Collapse Forcing} \label{subsect:collapse}

Whereas Cohen forcing creates a new injection \(\widecheck{\aleph_2} \hookrightarrow \mathcal{P}(\omega)\), we can use \emph{collapse forcing} to create a new surjection \(F : \aleph_1^{\mathbb{B}} \twoheadrightarrow \mathcal{P}(\omega)\).
Similarly to Cohen forcing, the strategy is to pick \lil{𝔹} such that there is a canonical \lil{𝔹}-valued indicator function on \(\widecheck{\aleph_1} \times \widecheck{\mathcal{P}(\omega)}\) representing the graph of a surjection \(\widetilde{F}\).
To show that \(\widetilde{F}\) suffices to force \(\CH\), we must verify that our choice of \(\mathbb{B}\) is \(\sigma\)-closed.

The formalization of collapse forcing is actually much more involved than the formalization of Cohen forcing. In Cohen forcing, we have to do relatively little work inside of \lil{bSet 𝔹} itself besides proving basic properties of functions. The difficulty is concentrated in proving and applying the CCC, which mostly happens in the metatheory. Moreover, constructing the new function (and the rest of the argument) required no density arguments at all. This is because in order to force \(\neg \CH\), we only had to ensure there was \emph{some} infinite cardinality between \(\omega\) and \(\mathcal{P}(\omega)\) (we did not determine exactly which internal aleph number \(\widecheck{\aleph_1}\) was in \lil{bSet 𝔹}).

However, to force \(\CH\), the quantifiers are flipped and now we must exclude \emph{all} cardinalities between \(\omega\) and \(\mathcal{P}(\omega)\). From cleverly choosing \(\mathbb{B}\), the best we can do is to construct a surjection \(\pi : \widecheck{\aleph_1} \twoheadrightarrow \widecheck{\mathcal{P}(\omega)}\), and we are forced to prove that \(\widecheck{\aleph_1} = \aleph_1^{\mathbb{B}}\) and \(\widecheck{\mathcal{P}(\omega)} = \mathcal{P}(\omega)\). This means we must define and construct \(\aleph_1^{\mathbb{B}}\), entailing, for example, the development of the theory of ordinals internal to \lstinline{bSet 𝔹}. For comparison, our library on set theory in \lstinline{bSet 𝔹} totalled 2723 LOC when we forced \(\neg \CH\), and grew to 7020 LOC after forcing \(\CH\).


\begin{defn}\label{def:collapse-poset}
  We define \(\mathbb{P}_{\mathsf{collapse}}\) to be the poset of countable partial functions \(\aleph_1 \to \mathcal{P}(\omega)\).
  The principal open sets
  \[D_p := \{g : \aleph_1 \to \mathcal{P}(\omega) \hspace{2mm} | \hspace{2mm} g \text{ extends } p\} , \hspace{3mm} p \in \mathbb{P}_{\mathsf{collapse}}\]
  form the basis of a topology \(\tau\) (finer than the product topology) on the function set \(\mathcal{P}(\omega)^{\aleph_1}\).
  We put
  \[\B_{\mathsf{collapse}} := \operatorname{RO}\left(\mathcal{P}(\omega)^{\aleph_1}, \tau\right).\]
\end{defn}

\begin{lemma}\label{lemma:collapse-algebra-sigma-closed}
  \(\B_{\mathsf{collapse}}\) is \(\sigma\)-closed.
\end{lemma}

\begin{proof}
  We show that the collection of principal open sets \(\mathcal{D} := \{D_p\}_p\) forms a dense subset of
  \(\B_{\mathsf{collapse}}\) such that every $\omega$-indexed downwards chain in \(\mathcal{D}\) has a lower bound in \(\mathcal{D}\). Since \(\mathcal{D}\) generates the topology, it is clearly a dense suborder. For an arbitrary \(\omega\)-indexed downwards chain
  \[D_{p_0} \supseteq D_{p_1} \supseteq \cdots \supseteq D_{p_n} \supseteq \cdots,\]
  it follows from the definition of the principal open sets that $p_0 \subseteq p_1 \subseteq \cdots \subseteq p_n \subseteq \cdots$. Then put $p_\omega := \bigcup_i p_i$. Since the union of countable partial functions is a countable partial function, $D_{p_\omega}$ is a lower bound of $\{D_{p_i}\}_i$.
\end{proof}

\begin{remark}
  As an implementation detail, in the formalization we \emph{define} \(\mathbb{P}_{\mathsf{collapse}}\) to be the countable partial functions (in \lil{Type u}) between \lil{(ordinal.mk (aleph one) : pSet).type} and \lil{(powerset omega : pSet).type}, so that \\
  \(\mathbb{B}_{\mathsf{collapse}}\)-valued indicator functions on
  \begin{gather*}\text{\lil{ordinal.mk (aleph one) : pSet).type ×}}\\
   \text{\lil{(powerset omega : pSet).type}}
  \end{gather*}
  are definitionally equal to \(\mathbb{B}_{\mathsf{collapse}}\)-valued indicator functions on the underlying types of \lil{check (ordinal.mk (aleph one))} and \lil{check (powerset omega)}.
\end{remark}

To specify the surjection \(\widecheck{\aleph_1} \twoheadrightarrow \widecheck{\mathcal{P}(\omega)}\), we specify a subset (the graph of the function) of the powerset \(\mathcal{P}(\widecheck{\aleph_1} \times \widecheck{\mathcal{P}(\omega)})\).
In \(\text{\lil{bSet}}\ \mathbb{B}_{\mathsf{collapse}}\), we can do this by specifying the indicator function \(\chi_{\pi}\) of the graph of a function \(\pi : \widecheck{\aleph_1} \to \widecheck{\mathcal{P}(\omega)}\) as follows: to an \(\eta < \aleph_1\) and a subset \(S \subseteq \mathcal{P}(\omega)\) (in \lil{pSet}), we attach the \emph{principal open} (comprising functions extending the singleton countable partial function \(\{(\eta, S)\}\)):
\[
  \chi_\pi (\eta, S) := D_{\{(\eta, S)\}} = \{g : \aleph_1 \to \mathcal{P}(\omega) \operatorname{|} g (\eta) = S\}.
\]

More generally, we formalize conditions over generic \lil{x, y : pSet} and \lil{𝔹} for when a function \lil{af : x.type → y.type → 𝔹} induces a surjection \(\widecheck{x} \to \widecheck{y}\) in \lil{bSet 𝔹}.
By definition, such a function always induces a relation on the product (in \lil{bSet 𝔹}) of \lil{x} and \lil{y}.
Surjectivity is equivalent to \lil{⨅ j, (⨆ i, af i j) = ⊤}, totality is equivalent to \lil{⨅ i, (⨆ j, af i j) = ⊤}, and well-definedness follows from conditions:
\begin{lstlisting}
(∀ i, ∀ j₁ j₂, j₁ ≠ j₂ → af i j₁ ⊓ af i j₂ ≤ ⊥)
(∀ i₁ i₂, ⊥ < (func x i₁) =ᴮ (func x i₂) → i₁ = i₂)
\end{lstlisting}
Both surjectivity and totality of \(\chi_{\pi}\) require \emph{density arguments}, where the definition of indexed supremum (\(\bigsqcup x_i\)) in the regular open algebra as the regularization \(((\bigcup x_i)^\perp)^\perp\) of the set-theoretic union plays a key role: the union of the truth values is not the entire space, but is only a dense open whose regularization is the entire space. In particular, the density argument for surjectivity crucially uses that \(\aleph_1\) is uncountable while \(\omega\) is countable.

To finish demonstrating that \(\CH\) is true in \(\text{\lil{bSet}}\ \mathbb{B}_{\mathsf{collapse}}\), it remains to check that \(\widecheck{\mathcal{P}(\omega)} = \mathcal{P}(\omega)\) and \(\widecheck{\aleph_1} = \aleph_1^{\mathbb{B}}\).
There are two major obstacles. The first is that to even formally state the latter equality, we must construct \(\aleph_1^{\mathbb{B}}\) in \lil{bSet 𝔹}.
While the operation \lil{bv_powerset} (\Cref{def:powerset}) gives a construction of the internal powerset of any \lil{x : bSet 𝔹} (using \lil{𝔹}-valued indicator functions, for any \lil{𝔹}), \(\aleph_1^{\mathbb{B}}\) is only specified as the least ordinal greater than \(\omega\), and does not admit as direct of a construction.
We describe our construction of \(\aleph_1^{\mathbb{B}}\) (as the Hartogs number of \(\omega\)) in \Cref{subsect:forcing:aleph-1}.

Now we must ensure that no new countable ordinals are added to \(\aleph_1\) and that no new subsets of \(\omega\) are added to \(\mathcal{P}(\omega)\) in the passage via \lil{check} from \lil{pSet} to \lil{bSet 𝔹}.
We show this in \Cref{subsect:function-reflection} by proving that we can reflect functions with domain \(\omega\) from \lil{bSet 𝔹} to \lil{pSet}.


\subsection{Construction of \texorpdfstring{$\aleph_1$}{aleph 1}} \label{subsect:forcing:aleph-1}
Instead of using the specification of \(\aleph_1^{\mathbb{B}}\) as the least ordinal larger than \(\omega\) with Cantor's theorem and using the well-foundedness of the ordinals to construct \(\aleph_1\),
we opt for a direct construction of \(\aleph_1\), based on the well-known construction of \(\aleph_1\) as the \textbf{Hartogs number} of \(\omega\) \cite{hartogs1915problem}.


We lay out the basic strategy.
Recall that a term of type \lil{bSet 𝔹} comprises three pieces of information: an indexing type \(\alpha\), an indexing function \lil{A : α → bSet 𝔹}, and a truth-value function \lil{B : α → 𝔹}.
\begin{enumerate}
\item We \emph{define} the underlying type \(\alpha\) for \(\aleph_1^{\mathbb{B}}\) to be \lil{𝒫(ω × ω).type}.
\item We define the truth-value function \lil{B : α → 𝔹} to assign to any \(R \subseteq \omega \times \omega\) the (truth-value of) the sentence,
``there exists an ordinal \(\eta\) and an injection \(f : \eta \hookrightarrow \omega\) such that \(R\) is the image of the membership relation of \(\eta\) under \(f\).''

\item Using the maximum principle (which is essentially \(\mathsf{AC}\)), we define the indexing function \(A\) for \(\aleph_1^{\mathbb{B}}\) by choosing, for every \(R : \alpha\), a witness \(\eta_R\) such that \(R\) is the image of \(\eta\) under an injection into \(\omega\).
That \(A\) surjects onto countable ordinals reduces to the fact that order-isomorphic ordinals must be equal.
\end{enumerate}

\paragraph{Implementation details}
In the formalization, this strategy is implemented in three stages.
First, the axiom of comprehension (\Cref{def:comprehension}) is applied to \(\mathcal{P}(\omega \times \omega)\) to produce (what \lil{bSet 𝔹} thinks is) the collection of all relations \(R\) on \(\omega\) such that \(B(R)\) holds.
This combines steps \(1\) and \(2\) and produces a set \lil{a1'_aux}. Then we \emph{modify} the indexing function \lil{a1'_aux.func} (by using the maximum principle
) to point from \(R\) to a chosen witness \(\eta_R\) for \(R\), producing \lil{a1'}.
Finally, since the ordinals \(0\) and \(1\) both have empty membership relations, it is unprovable in Lean whether \lil{a1'} contains one or the other, so we add both manually, producing \(\aleph_1^{\mathbb{B}}\).

Our implementation differs from the usual construction of Hartogs numbers by \emph{starting} with the sub-well-orders of \(\omega\), rather than taking the class of countable ordinals and later showing it is a set. In this way we avoid performing a smallness argument, at the cost of using the axiom of choice to select witnesses. We remark that our construction does not use specific properties of \(\omega\) and easily generalizes to construct the successor cardinal of any infinite set. Instead of using membership \((<)\), we could have used subset \((\leq)\) instead, which would avoid the intermediate \lil{a1'}, but this would have made other parts of the proof more complex.


\subsection{Function Reflection} \label{subsect:function-reflection}

Suppose given \lil{y : pSet} and \lil{f : bSet 𝔹} such that \lil{bSet 𝔹} models that \lil{f} is a function from \(\omega\) to \(\widecheck{y}\).
We say that \lil{bSet 𝔹} \textbf{reflects \lil{f}} if there exists a \lil{g : pSet} such that \lil{g} is a function from \(\omega\) to \lil{y} in \lil{pSet}, and \lil{bSet 𝔹} models that \(\widecheck{g} = f\).
We say that \lil{bSet 𝔹} \textbf{reflects countable functions} if it reflects all such \lil{f}.
\begin{lemma}\label{lemma:function-reflect-suffices}
  Let \(\mathbb{B}\) be a complete Boolean algebra, and suppose that \lil{bSet 𝔹} reflects countable functions.
  Then \(\widecheck{\mathcal{P}(\omega)} = \mathcal{P}(\omega)\) and \(\widecheck{\aleph_1} = \aleph_1^{\mathbb{B}}\).
\end{lemma}
\begin{proof}

  To see that \(\widecheck{\aleph_1} \subseteq \aleph_1^{\mathbb{B}}\), let \(x\) be an arbitrary element of \(\widecheck{\aleph_1}\).
  By definition \(x\) is equal to \(\widecheck{\eta}\) for some \(\eta < \aleph_1\) in \lil{pSet}.
  Since the ordinals and cardinals of \lil{pSet} are isomorphic to Lean's ordinals and cardinals for \lil{Type u}, \(\eta\) injects into \(\omega\) (in \lil{pSet}, and also at the level of indexing types).
  Since being an injective function is \(\Delta_0\), it is absolute for \lil{check}, so \(x = \widecheck{\eta}\) injects into \(\omega\).
  Then, by definition of \(\aleph_1^{\mathbb{B}}\) we have \(x \in \aleph_1^{\mathbb{B}}\).\footnote{Note that this did not use our assumption, and holds for general \(\mathbb{B}\).
  For a conventional proof in a set-theoretic metatheory, see e.g. \cite{bell2011set}}

  To see that \(\aleph_1^{\mathbb{B}} \subseteq \widecheck{\aleph_1}\), suppose towards a contradiction that this is not true; since the ordinals are well-ordered,
  this means that \(\widecheck{\aleph_1} < \aleph_1^{\mathbb{B}}\), so by definition of \(\aleph_1^{\mathbb{B}}\), there is a surjection \(f : \omega \to \widecheck{\aleph_1}\).
  By assumption, this surjection can be lifted to a function \(g : \omega \to \aleph_1\) in \lil{pSet}, which can again be checked to be surjective, a contradiction.

  Similarly, it is true for general \(\mathbb{B}\) and any \lil{x : pSet} that \(\widecheck{\mathcal{P}(x)} \subseteq \mathcal{P}(\widecheck{x})\),
  because indicator functions into \lil{bool} naturally induce indicator functions to \(\mathbb{B}\) (by composing with the canonical inclusion \lil{bool → 𝔹}).
  Conversely, to show that \(\mathcal{P}(\omega) \subseteq \widecheck{\mathcal{P}(\omega)}\), use the isomorphism \(\mathcal{P}(\omega) \simeq \widecheck{2}^{\widecheck{\omega}}\) to reduce this to showing that \(\widecheck{2}^{\widecheck{\omega}} \subseteq \widecheck{2^{\omega}}\),
  and then apply the assumption to an arbitrary element of \(\widecheck{2}^{\widecheck{\omega}}\).
\end{proof}

It remains to show that \(\mathbb{B}_{\mathsf{collapse}}\) fulfills the assumptions of \Cref{lemma:function-reflect-suffices}.

\begin{lemma}\label{lemma:function-reflect}
  \(\text{\lil{bSet}}\ \mathbb{B}_{\mathsf{collapse}}\) reflects countable functions.
\end{lemma}


\begin{proof}
  Fix \(y\) and \(f\). It suffices to show that
\begin{lstlisting}
f ∈ᴮ functions ω y̌
  ≤ ⨆ (g : bSet 𝔹),
      g ∈ᴮ (functions omega y)̌  ⊓ g =ᴮ f
\end{lstlisting}
  and by a density argument, it suffices to show that for every principal open \(D_p\), for \(D := D_p \cap f \in^{\B} \mathsf{functions}\ \omega \ \widecheck{y}\),
\begin{lstlisting}
⊥ < (⋃ g, D ⊓ g ∈ᴮ (functions omega y)̌  ⊓ g =ᴮ f)
\end{lstlisting}
  It suffices to construct a single function \(g : \omega \to y\) such that \(\bot < D \sqcap \widecheck{g} = f\).
  As with Cohen forcing, we will reflect a Boolean-valued \(\forall\)-\(\exists\) statement into the metatheory, and then use a combinatorial property of \(\mathbb{B}_{\mathsf{collapse}}\) to strengthen it.
  The following lemma is true for general \(\mathbb{B}\):
\begin{lstlisting}
lemma AE_of_check_func_check (x y : pSet)
  {f : bSet 𝔹} {Γ : 𝔹}
  (H : Γ ≤ is_func' x̌ y̌ f) (H_nonzero : ⊥ < Γ) :
  ∀ (i : x.type), ∃ (j : y.type) (Γ' : 𝔹)
  (H_nonzero' : ⊥ < Γ') (H_le : Γ' ≤ Γ),
  Γ' ≤ (is_func' x̌ y̌ f) ∧
  Γ' ≤ (pair (x.func i)̌  (y.func j)̌ ) ∈ᴮ f
\end{lstlisting}
Recursively applying this lemma, we obtain \(g_0, \dots, g_n, \dots\) such that

\[D \sqcap (0 , g_0) \in^{\mathbb{B}} f > \cdots > D \sqcap \left(\bigsqcap_{k \leq n} ((k, g_k) \in^{\mathbb{B}} f)\right) > \cdots > \bot. \]

The lower bound of this chain implies that the required lift of \(f\) is
\(g := \{(k, g_k)\}_{k \in \omega}\).
For general \(\mathbb{B}\), this lower bound might be \(\bot\), but because \(\mathbb{B}_{\mathsf{collapse}}\) is \(\sigma\)-closed,
we can shrink each term of the above chain into a dense suborder \(\mathcal{D}\) such that all downward \(\omega\)-indexed chains in \(\mathcal{D}\) have nonzero intersection, so the intersection of the chain is indeed nonzero.
\end{proof}

Implementing this argument was one of the most technical parts of our formalization. At each step of the construction of the downwards chain, we must recursively apply a \(\forall\)-\(\exists\) statement and use the axiom of choice to select two witnesses (with four side conditions), which are then used to simultaneously construct the downwards chain and the function \lstinline{g : pSet}. This was implemented as a monolithic recursive function defined using Lean's equation compiler, with the required parts separated afterwards.



\subsection{The Independence of \texorpdfstring{\(\CH\)}{CH}} \label{subsect:forcing:independence}

In \Cref{subsect:bset:fundamental-thm} we showed that \lil{bSet 𝔹} is a model of \lil{ZFC},
which means that we can interpret the deeply-embedded statement of \lil{CH_formula} into
\lil{bSet 𝔹}. It is easy to verify that the deeply-embedded interpretation of \lil{CH_formula} coincide with the shallow interpretations of \(\mathsf{CH}\).

As we have already observed, an easy consequence of Boolean-valued soundness is that a formula is unprovable if its negation has a model. Thus, we have:
\begin{lstlisting}
lemma unprovable_of_model_neg {C : Theory L}
  {f : sentence L} (S : bStructure L 𝔹)
  (H_model : ⊤ ⊩[S] C) [H_nonempty : nonempty S]
  {Γ : 𝔹} (H_nonzero : (⊥ : 𝔹) < Γ)
  (H : Γ ⊩[S] ∼f) : ¬ (C ⊢' f)
lemma V_𝔹_cohen_models_neg_CH :
  ⊤ ⊩[V 𝔹_cohen] ∼CH_formula
lemma V_𝔹_collapse_models_CH :
  ⊤ ⊩[V 𝔹_collapse] CH_formula
\end{lstlisting}
\noindent Combining these results yields
\begin{lstlisting}
theorem CH_unprv : ¬ (ZFC ⊢' CH_formula)
theorem neg_CH_unprv : ¬ (ZFC ⊢' ∼CH_formula)
\end{lstlisting}
\noindent and the independence of CH follows.
\begin{lstlisting}
def independent (T : Theory L) (f : sentence L) :=
¬ T ⊢' f ∧ ¬ T ⊢' ∼f
theorem independence_of_CH : independent ZFC CH_f :=
by finish [independent, CH_unprv, neg_CH_unprv]
\end{lstlisting}
\section{Automation and Metaprogramming}
\label{sect:metaprogramming}

A key feature of Lean is that it is its own metalanguage \cite{Ebner:2017:MFF:3136534.3110278}, allowing for seamless in-line definitions of custom tactics (and modifications of existing ones).
This was an invaluable asset, allowing us to rapidly develop a custom tactic library for simulating natural-deduction style proofs in complete Boolean algebras (\Cref{subsect:natded}) and automating equality reasoning in those proofs (\Cref{subsect:bv-cc}).

\subsection{Simulating Natural Deduction Proofs in Complete Boolean Algebras} \label{subsect:natded}
As stressed by Scott \cite{scott2008algebraic}, ``A main point ... is that the well-known algebraic characterizations of [complete Heyting algebras] and [complete Boolean algebras] exactly mimic the rules of deduction in the respective logics.''
Indeed, that is really why the Boolean-valued soundness theorem (see \Cref{boolean-soundness}) is true: one can just replay natural deduction proofs in arbitrary complete Boolean algebras, not just \lil{Prop}. We use Lean's metaprogramming to expose natural deduction-style tactics to the user for the purpose of proving inequalities in complete Boolean algebras. (One thinks of the \lil{≤} symbol in an inequality of Boolean truth-values as a turnstile in a proof state). An immediate challenge which arises is being able to reason about assumptions (to the left of the turnstile) modulo associativity and commutativity. For example, the natural-deduction version of this statement should simply be \lil{by assumption}:
\begin{lstlisting}
∀ a b c d e f g : 𝔹,
  (d ⊓ e) ⊓ (f ⊓ g ⊓ ((b ⊓ a) ⊓ c)) ≤ a
\end{lstlisting}
but with a naive approach, one must manually unwrap and permute the arguments of the nested \(\sqcap\)s. Our solution is to piggyback on the tactic monad's AC-invariant handling of hypotheses in the tactic state, by applying the \emph{Yoneda lemma} for posets:
\label{poset-yoneda}
\begin{lstlisting}
lemma poset_yoneda {β} [partial_order β] {a b : β}
  (H : ∀ Γ : β, Γ ≤ a → Γ ≤ b) : a ≤ b
\end{lstlisting}
With a little custom automation, our first example nearly becomes ``\lil{by assumption}''
\begin{lstlisting}
example {a b c d e f g : 𝔹} :
  (d ⊓ e) ⊓ (f ⊓ g ⊓ ((b ⊓ a) ⊓ c)) ≤ a :=
by { tidy_context, assumption }
/-  Goal state before `assumption`:
  [...]
  H_right_right_left_left : Γ ≤ b,
  H_right_right_left_right : Γ ≤ a
  ⊢ Γ ≤ a  -/
\end{lstlisting}
In this example, \lil{tidy_context} combines an application of \lil{poset_yoneda} with a call to the simplifier to split hypotheses of the form \lstinline{Γ ≤ a₁ ⊓ a₂ ⊓ ... aₙ} into \lstinline{Γ ≤ a₁, Γ ≤ a₂, ..., Γ ≤ aₙ}.
With more sophisticated tricks, such as coercing assumptions of the form \lil{(Γ ≤ a ⟹ b)} to functions \lil{Γ ≤ a → Γ ≤ b}, automated propagation of change-of-variables (``context-specialization'', see \cite{DBLP:conf/itp/HanD19} for more details), and automatically casing on disjunctions \lil{Γ ≤ a ⊔ b}, it is even possible to write a Boolean-valued tableaux prover \lil{bv_tauto}:
\begin{lstlisting}
example {a b c : 𝔹} :
  (a ⟹ b) ⊓ (b ⟹ c) ≤ a ⟹ c :=
  by { tidy_context, bv_tauto }
\end{lstlisting}
Compare this with a more conventional proof, where we even have the deduction theorem and modus ponens available as lemmas: 
\begin{lstlisting}
example {β : Type*} [complete_boolean_algebra β]
  {a b c : β} :
 ( a ⟹ b ) ⊓ ( b ⟹ c ) ≤ a ⟹ c :=
begin
  rw [ ← deduction, inf_comm, ← inf_assoc ],
  transitivity b ⊓ (b ⟹ c),
    { refine le_inf _ _,
      { apply inf_le_left_of_le, rw inf_comm,
        apply mp },
      { apply inf_le_right_of_le, refl }},
    { rw inf_comm, apply mp }
end
\end{lstlisting}

It would have been possible to go further and even write a custom tactic state,
as was done for temporal logic in Unit-B \cite{Hudon2015TheUM} or for Lean's SMT-mode framework,
such that the machinery for handling the ambient context \(\Gamma\) is completely hidden. 
However, we judged the benefits of this to be mostly cosmetic, and we leave more sophisticated implementations for future work.

\subsection{Boolean-valued Equality Reasoning}

\paragraph{Congruence Closure on Quotient Types} \label{subsect:bv-cc}
Another benefit of applying \hyperref[poset-yoneda]{\lstinline{poset_yoneda}} and using context variables \(\Gamma\) throughout the formalization is that this approach exposes a canonical poset of setoids on \lil{bSet 𝔹} induced by \lil{𝔹}-valued equality:
for every \(\Gamma : \mathbb{B}\) the relation \(\lambda\; x \; y, \Gamma \leq x =^{\mathbb{B}} y\) is an equivalence relation on \lil{bSet 𝔹}. 

Since Lean natively supports quotient types, then as soon as the only task remaining is to perform equality reasoning,
we can quotient by the appropriate setoid and simply call \lil{cc};
this is easy to automate with a custom tactic \lil{bv_cc}. 
We can add support for any predicate satisfying an appropriate \(\mathbb{B}\)-valued congruence lemma,
although we currently add support for individual predicates by hand:
\begin{lstlisting}
example {x₁ y₁ x₂ y₂ : bSet 𝔹} {Γ}
  (H₁ : Γ ≤ x₁ ∈ᴮ y₁) (H₂ : Γ ≤ x₁ =ᴮ x₂)
  (H₂ : Γ ≤ y₁ =ᴮ y₂) : Γ ≤ x₂ ∈ᴮ y₂ := by bv_cc
\end{lstlisting}
\paragraph{Discharging Congruence Lemmas} \label{subsect:B-ext}
Rewriting along a \(\mathbb{B}\)-valued equality is the same as rewriting in the appropriate setoid parametrized by the current context \(\Gamma\), 
so that the motive must satisfy an appropriate \emph{congruence lemma} \lil{h_congr} with respect to the equivalence relation:
\begin{lstlisting}
lemma bv_rw {x y : bSet 𝔹} {Γ : 𝔹}
  (H : Γ ≤ x =ᴮ y) {ϕ : bSet 𝔹 → 𝔹}
  {h_congr : ∀ x y, x =ᴮ y ⊓ ϕ x ≤ ϕ y}
  {H_new : Γ ≤ ϕ y} : Γ ≤ ϕ x
\end{lstlisting}
We alias the type of \lil{h_congr}, and add a database of \lil{@[simp]} lemmas expressing that congruence lemmas are preserved by first-order logical operations:
\begin{lstlisting}
def B_ext (ϕ : bSet 𝔹 → 𝔹) : Prop :=
∀ x y, x =ᴮ y ⊓ ϕ x ≤ ϕ y
@[simp] lemma B_ext_infi {ι} {ϕ : ι → (bSet 𝔹 → 𝔹)}
  {h : ∀ i, B_ext (ϕ i)} : B_ext (λ x, ⨅i, ϕ i x)
\end{lstlisting}
Furthermore, \lil{simp} is able to handle recursive applications of these lemmas on its own, allowing most congruence lemma proof obligations to be automatically discharged:
\begin{lstlisting}
example {w : bSet 𝔹} :
  (let ϕ := λ x, ⨅ z, z ∈ᴮ w ⊓ z ⊆ᴮ x ⊓ x ⊆ᴮ z
  in B_ext ϕ) := by simp
\end{lstlisting}
\section{Conclusions and Future Work}
\label{sect:conclusions}

Interestingly, we never used transfinite recursion for developing elementary set theory in \lil{pSet} and \lil{bSet 𝔹}. Indeed, the prevalence of transfinite recursion in traditional presentations of set theory is only a consequence of the use of transfinite recursion in the traditional definitions of \(V\) and \(V^{\mathbb{B}}\). By instead encoding \(V\) and \(V^{\mathbb{B}}\) as inductive types which expose \(\in\)-induction as their native induction principle, we completely eliminate transfinite induction from this part of our formalization. 

Our consistency proof of \(\CH\) is very different from the traditional proof, due to G\"odel, which shows that the constructible universe \(\mathsf{L}\) satisfies \(\mathsf{GCH}\). An obvious path to constructing \(\mathsf{L}\) is to define the definable powerset operation with an inductive predicate on \lil{pSet} whose constructors encode the nine G\"odel operations, and to then build the constructible hierarchy by transfinite recursion. It is interesting to consider whether there is a definition of \(\mathsf{L}\) in the same spirit as \lil{pSet} which completely avoids transfinite induction.

We also want to formalize the conservativity of \(\mathsf{ZFC}\) over the usual presentation in the language \(\{\in\}\), by proving more generally that extending a language with definable function symbols is conservative.
Furthermore, while formulas with de Bruijn indices enjoy pleasant theoretical properties, they are difficult to write and debug by hand. It should be possible with Lean's metaprogramming to write a custom parser from formulas with named variables.

Although our custom automation saved a considerable amount of work, much of it is only an approximation to a more principled approach by \emph{reflection}.
The natural deduction and equality reasoning tactics in \Cref{subsect:natded} and \Cref{subsect:bv-cc} make it easier to manually replay a first-order proof of a theorem of \(\ZFC\) in \lil{bSet 𝔹},
but the Boolean-valued soundness theorem automatically performs this replay for a deeply-embedded first-order proof tree.
Ideally, automation would reify a \(\mathbb{B}\)-valued goal to the corresponding first-order statement, discharge it by an ATP, encode the solution in our deeply-embedded proof system, then apply soundness.
Alternately, one could perform proof transfer via the completeness theorem, proving a first-order goal in an arbitrary ordinary model of \(\ZFC\) first
, then applying \(\mathbb{B}\)-valued soundness to the proof tree gotten by completeness.
The advantage to this approach is that a proof would only be computed once, then reused in any model, ordinary or \(\mathbb{B}\)-valued, whereas in our formalization, we occasionally had to prove the same statement separately in \lil{pSet} and \lil{bSet 𝔹}.






Besides the construction of \(\mathsf{L}\), the consistency of $\mathsf{GCH}$ can also be shown by an iterated forcing argument. Our current implementation of forcing should extend without too much difficulty to iterated forcing with Boolean-valued models. There are also many generalizations of the consistency of $\neg \mathsf{CH}$.
An interesting challenge could be Easton's theorem, which states that on regular cardinals the function $\kappa\mapsto 2^\kappa$ can be any monotone function not contradicting K\"onig's Theorem ($\kappa<\cf(2^\kappa)$)~\cite{easton1970powers}.

Our work only marks the beginning of an integration of formal methods with modern set theory. Since Cohen, increasingly sophisticated forcing arguments have been used to produce a vast hierarchy of independence and relative consistency results. The challenge to proof engineers is to develop libraries and automation that can uniformly handle them, so that the manipulation of forcing notions and forcing extensions in a proof assistant becomes as routine as manipulating objects in an algebraic hierarchy is today. One place to start would be to develop a good interface for forcing with posets, and for transferring arguments along the equivalence to Boolean-valued models. One could develop a typeclass hierarchy of combinatorial conditions on forcing notions, and similarly for the relative consistency strengths of extensions to \(\ZFC\).
As the next challenge to formalizers, we propose the classical result of Shelah \cite{shelah1974infinite} on the independence of Whitehead's problem, the proof of which combines the consistency of the \(\mathsf{ZFC} + (\mathsf{V} = \mathsf{L})\) with the consistency of Martin's axiom \cite{martin1970internal} over \(\ZFC + \neg \CH\) to resolve a conjecture in abstract algebra.



\begin{acks}                            
  We thank the members of the CMU-Pitt Lean group, particularly Simon Hudon, Jeremy Avigad, Mario Carneiro, Reid Barton, and Tom Hales for their feedback and suggestions; we are also grateful to Dana Scott and John Bell for their advice and correspondence.

  The authors gratefully acknowledge the support by the
  \grantsponsor{GS100000001}{Alfred P. Sloan Foundation}{https://doi.org/10.1038/201765d0}, Grant
  No.~\grantnum{GS100000001}{G-2018-10067}.
\end{acks}

\bibliography{flypitch-cpp}




\end{document}